\documentclass{amsart}
\usepackage[english]{babel}
\usepackage{amsrefs}
\usepackage{amssymb,amsmath,amsthm,amsfonts,epsfig,graphicx,psfrag,color}
\usepackage[utf8]{inputenc}
\usepackage{amsrefs}
\usepackage{hyperref}
\hypersetup{
    colorlinks,
    citecolor=black,
    filecolor=blue,
    linkcolor=black,
    urlcolor=blue
}

\theoremstyle{plain}
\newtheorem{teo}{Theorem}[section]
\newtheorem{thm}[teo]{Theorem}
\newtheorem{cor}[teo]{Corollary}
\newtheorem{lem}[teo]{Lemma}
\newtheorem{prop}[teo]{Proposition}		

\theoremstyle{definition}
\newtheorem{df}[teo]{Definition}
\newtheorem{exa}[teo]{Example}
\newtheorem{rmk}[teo]{Remark}

\DeclareMathOperator{\diam}{diam}

\DeclareMathOperator{\dist}{dist}

\newcommand{\rep}{\mathcal{H}^+(\R)}
\newcommand{\repo}{\mathcal{H}^+_0(\R)}

\DeclareMathOperator{\fix}{fix}

\newcommand{\R}{\mathbb R}

\newcommand{\Z}{\mathbb Z}

\renewcommand{\epsilon}{\varepsilon}

\begin{document}

\author{Alfonso Artigue}
\address{Departamento de Matemática y Estadística del Litoral, Universidad de la Rep\'ublica, 
Gral. Rivera 1350, Salto, Uruguay.}
\email{artigue@unorte.edu.uy}
\title{Rescaled expansivity and separating flows}
\date{\today}
\begin{abstract} 
In this article we give sufficient conditions for Komuro expansivity to imply the
rescaled expansivity recently introduced by Wen and Wen. 
Also, we show that a flow on a compact metric space is expansive in the sense of Katok-Hasselblatt if and only if it is separating in the sense of Gura 
and the set of fixed points is open.
\end{abstract}
\maketitle

\section{Introduction}
The study of expansive flows started in 1972 with the works of Bowen, Walters and Flinn \cite{BW,Flinn}. 
See \S \ref{secBWExp} for the definition.
Since then, several results have been translated from dynamics of discrete time (i.e., expansive homeomorphisms) to flows without fixed points. 
In \cite{KS} Keynes and Sears extended results of Mañé proving that a space 
admitting an expansive systems has finite topological dimension. 
Also, they proved (under certain conditions) that if a space admits a minimal expansive flow then the space is one-dimensional.
In \cite{IM,Pa} Inaba and Matsumoto and, independently, Paternain extended results of Hiraide and Lewowicz of expansive homeomorphisms on surfaces to flows on three-dimensional manifolds. 
Other results of expansive three-dimensional flows were developed in \cite{Bru93} by Brunella.
In \cite{MSS} Moriyasu, Sakai and Sun characterized the $C^1$-robustly expansive vector fields by the condition of being quasi-Anosov, 
extending another result of Mañé for diffeomorphisms. 
More recently, in \cite{Willy} Cordeiro introduced the notion of cw-expansive flow, 
and proved that if the topological dimension of the space is greater than 1 then the topological entropy of a cw-expansive flow on this space is positive (which extends 
a previous result by Kato).
In light of these results, we can fairly say that the definition of expansivity of \cite{BW,Flinn} 
is an adequate way to translate expansivity from homeomorphisms to flows without fixed points, as it gives a reasonably parallel theory.

For the general theory of dynamical systems, the attractor discovered by Lorenz in 1963 presented several challenges that could not be understood in the framework of hyperbolicity, in the usual and uniform sense.
A special characteristic of this attractor is the presence of a fixed point. From a topological viewpoint the main properties of a hyperbolic set are expansivity and 
the possibility of shadowing pseudo-orbits. 
In the study of the \emph{expansive properties} of the Lorenz attractor, 
Komuro proved that it is $k^*$-expansive, a notion introduced in \cite{K}.
To put this particular dynamics in a more general theory, 
in \cite{APPV} Araujo, Pacifico, Pujals and Viana considered 
singular hyperbolic attractors, 
and proved that this kind of systems are $k^*$-expansive, generalizing the result of Komuro that we mentioned.
Singular hyperbolicity is a property that captured the essence of the Lorenz attractor specially on three-manifolds.
On higher dimensional manifolds new phenomena can occur as the existence of two fixed points with different indices on a common recurrence class. 
Recently, Bonatti and da Luz \cite{BodL} introduced the notion of multisingular hyperbolicity to provide a general framework 
to the understanding of singular flows on manifolds of arbitrary dimension.
Naturally, it is of interest to determine the \emph{expansive properties} implied by multisingular hyperbolicity. 
It is not known whether multisingular hyperbolic sets are $k^*$-expansive.
In \cite{WW} Wen and Wen introduced a new form of expansivity, called \emph{rescaled expansivity}, and they proved that 
multisingular hyperbolicity implies this definition. 
In that paper the authors ask whether there are connections between $k^*$-expansivity and their rescaled version.

In \S \ref{secKomYRes} of the present paper we will study rescaling expansive flows. 
Our main result is Theorem \ref{thmKestImpRes},
where we will give sufficient conditions for a $k^*$-expansive flow to be rescaling expansive. 
The condition we found, that we call \emph{efficiency} (see \S \ref{secEfficient}), 
is satisfied (for instance) if the fixed points of the flow are hyperbolic, Corollary \ref{corRescTiempo}. 
The notion of efficiency is introduced to have a control of the distance of the extreme points of a small orbit segment 
over the diameter of the segment. We give examples showing that rescaled expansivity does not imply $k^*$-expansivity.
Also, we consider a rescaled Riemannian metric induced by the original Riemannian metric and the velocity field of the flow. 
We will show that BW-expansivity with respect to the rescaled metric implies rescaled expansivity.

We will also prove some relations between other versions of expansivity that can be found in the literature, 
and we give \emph{simpler} characterizations.
In \S \ref{secSepFlowMain} we consider a weak form of expansivity, introduced by Gura \cite{Gura} and called \emph{separating}. 
It is defined without reparameterizations of the orbits. 
In this section we clarify the meaning of the definition of expansive flow in the Katok-Hasselblatt book \cite{HK}.
In Theorem \ref{thmHKExp} we show that a flow of a compact metric space 
is expansive in the sense of \cite{HK} if and only if it is separating and the set of fixed points is open.

In \S \ref{secFlowOfSurf} we consider flows of compact surfaces. We show that every strongly separating and efficient flow is rescaling expansive. 
In \S \ref{secExSph} we study a time change of the \emph{north-south flow} on the two-sphere. 
We show that rescaled expansivity is not invariant under time changes of the flow and that 
a rescaling expansive flow may not be strongly separating.

\section{Separating flows}
\label{secSepFlowMain}
Let $\phi\colon \R\times \Lambda\to\Lambda$ be a continuous flow of the metric space $(\Lambda,\dist)$.
As fixed points will play a key role in the paper, in \S \ref{secFixP} we start deriving some of their basic properties.
In \S \ref{secSepFlow} we consider separating flows, as defined by Gura, and HK-expansivity, a way to define expansivity for flow from \cite{HK}.

\subsection{Fixed points} 
\label{secFixP}
The set of \emph{fixed} points of $\phi$ is defined as
\[
\fix(\phi)=\{\sigma\in\Lambda:\phi_t(\sigma)=\sigma\text{ for all }t\in\R\}.
\]
A point that is not fixed is called \emph{regular}.
The continuity of $\phi$ implies that $\fix(\phi)$ is closed in $\Lambda$.
The next result is an extension of \cite[Lemma 2]{BW}. 

\begin{lem}
\label{lemFixOpen}
For a flow $\phi$ on a compact metric space $(\Lambda,\dist)$ the following statements are equivalent: 
\begin{itemize}
 \item $\fix(\phi)$ is an open subset of $\Lambda$,
 \item there is $T_0>0$ such that for all 
$T\in (0,T_0]$ there is $\xi>0$ such that $\dist(\phi_T(x),x)>\xi$ for all $x\notin\fix(\phi)$.
\end{itemize}
\end{lem}

\begin{proof}
\emph{Direct}. Arguing by contradiction, suppose that there is 
a sequence $T_n\to 0$, $T_n>0$, and 
for each $n\geq 1$ there is a sequence $\xi^n_m>0$ such that 
$\xi^n_m\to 0$ as $m\to\infty$ and 
$\dist(x^n_m,\phi_{T_n}(x^n_m))<\xi^n_m$ for some $x^n_m\in \Lambda\setminus \fix(\phi)$. 
Since $\Lambda$ is compact we assume that $x^n_m\to y_n\in \Lambda$ as $m\to\infty$. 
Given that $\fix(\phi)$ is open, we have that $y_n\notin\fix(\phi)$.
Suppose that $y_n\to z\in\Lambda$. 
By the continuity of the flow, $\phi_{T_n}(y_n)=y_n$ for all $n\geq 1$. 
Since $T_n\to 0$, given $\tau\in\R$ there is a sequence $k_n\in\Z$ such that $k_nT_n\to\tau$. 
Then $\phi_{k_nT_n}(y_n)=y_n$ and, taking $n\to\infty$, we conclude that $\phi_\tau(z)=z$. 
As $\tau\in\R$ is arbitrary, we have that $z\in\fix(\phi)$. 
This contradicts that $\fix(\phi)$ is open because $z$ is accumulated by $y_n\notin\fix(\phi)$.

\emph{Converse}. 
Suppose that $x_n\in\Lambda\setminus\fix(\phi)$ with $x_n\to y\in\Lambda$. 
Since $\dist(\phi_T(x_n),x_n)>\xi$ for all $n\geq 1$, taking limit we have that $\dist(\phi_T(y),y)\geq\xi>0$ and $y\notin\fix(\phi)$. 
Then $\Lambda\setminus\fix(\phi)$ is closed and $\fix(\phi)$ is open.
\end{proof}

The orbit of $x$ will be denoted as $\phi_\R(x)$.
We say that $\sigma\in\fix(\phi)$ is \emph{dynamically isolated} if there is $r>0$ such that $\phi_\R(x)\subset B(\sigma,r)$ implies $x=\sigma$,
where $B(x,a)$ denotes the open ball of radius $a$ centered at $x$. 

\begin{lem}
\label{lemIsoBetaCero}
 If $\Lambda$ is compact and each $\sigma\in\fix(\phi)$ is dynamically isolated then 
 there is $\beta_0>0$ such that if $\diam(\phi_\R(x))<\beta_0$ then $x\in\fix(\phi)$.
\end{lem}

\begin{proof}
 Arguing by contradiction suppose that there is a sequence $x_n$ such that $\diam(\phi_\R(x_n))>0$ and converges to 0. 
 Since $\Lambda$ is compact we can assume that $x_n$ has a limit point $\sigma$. The continuity of the flow implies that $\sigma$ is a fixed point. 
 Given that $\diam(\phi_\R(x_n))>0$ we see that $x_n\neq\sigma$, and on every neighborhood of $\sigma$ we have whole orbits, not containing $\sigma$, which contradicts that 
 $\sigma$ is dynamically isolated.
\end{proof}

\subsection{Separating flows} 
\label{secSepFlow}
In this section we consider two definitions of expansivity. 
The first one does not consider parameterizations of trajectories.

\begin{df}[\cite{Gura}] 
We say that $\phi$ is \emph{separating} if there is $\delta>0$ such that if 
\[
\dist(\phi_t(x),\phi_t(y))<\delta\text{ for all }t\in\R 
\]
then $y\in\phi_\R(x)$. In this case we say that $\delta$ is a \emph{separating constant}.
\end{df}

\begin{rmk}
\label{rmkSepFix}
From the definition it follows that if $\phi$ is separating 
then 
each fixed point is dynamically isolated and
if the space is compact we conclude that $\fix(\phi)$ is a finite set. 
\end{rmk}

The second definition of this section does consider reparameterizations 
but in a very particular way.

\begin{df}[\cite{HK}]
\label{dfHKExp}
We say that $\phi$ is \emph{KH-expansive}
if there is $\delta>0$ such that:
if $x\in \Lambda$, $s\colon \R\to\R$ is continuous, $s(0)=0$
and $\dist(\phi_t(x),\phi_{s(t)}(x))<\delta$ for all $t\in\R$, 
$y\in \Lambda$ is such that $\dist(\phi_t(x),\phi_{s(t)}(y))<\delta$ 
for all $t\in\R$ then $y\in\phi_\R(x)$.
\end{df}

\begin{prop}
\label{propHkFixOpen}
 If $\phi$ is KH-expansive then $\fix(\phi)$ is open. 
\end{prop}

\begin{proof}
 Suppose that $p\in\fix(\phi)$ is accumulated by regular points $x_n\notin\fix(\phi)$, $x_n\to p$. 
Take $x_n\in B(p,\delta)$ and $s(t)=0$ for all $t\in\R$. 
Then, $\dist(\phi_t(p),\phi_{s(t)}(p))=0$ for all $t\in\R$ and
$\dist(\phi_t(p),\phi_{s(t)}(x_n))<\delta$ for all $t\in\R$. 
Since $x_n\notin\phi_\R(p)$ we conclude that $\phi$ is not KH-expansive. 
\end{proof}

\begin{prop}
\label{propHKImpSep}
 Every KH-expansive flow is separating. 
\end{prop}

\begin{proof}
To see that $\phi$ is separating we take $s(t)=t$ for all $t\in\R$ in the definition. 
\end{proof}

In general, the converse of Proposition \ref{propHKImpSep} is not true: 

\begin{exa}
\label{exaToroUnaSing}
Consider the two-dimensional torus $\Lambda=\R^2/\Z^2$ with its usual flat Riemannian metric. 
Let $X$ be a constant vector field on $\Lambda$ generating a minimal flow (every orbit is dense in $\Lambda$).  
Assume that $\|X(x)\|=1$ for all $x\in \Lambda$.
Take a smooth function $\rho\colon \Lambda\to[0,1]$ vanishing only at a given point $\sigma\in \Lambda$. 
Let $\phi$ be the flow induced by $\rho X$. 
It is not KH-expansive because $\fix(\phi)=\{\sigma\}$ is not open in $\Lambda$. 
The flow $\phi$ is separating, 
moreover, it is strong kinematic expansive, see \cite[Example 2.8 and Theorem 3.8]{Ar}.
\end{exa}

The next result shows that the converse of Proposition \ref{propHKImpSep} is true if we assume that the set of fixed points is open and the space is compact. 

\begin{thm}
\label{thmHKExp}
A flow $\phi$ of a compact metric space $(\Lambda,\dist)$ is KH-expansive if and only if it is separating and $\fix(\phi)$ is open.
\end{thm}

\begin{proof}
The direct part follows by Propositions \ref{propHkFixOpen} and \ref{propHKImpSep}. 
To prove the converse, suppose that $\phi$ is separating and $\fix(\phi)$ is open. 
Let $\alpha$ be a separating constant. 
Since $\Lambda$ is compact we have that $\fix(\phi)$ is finite and there is
$r>0$ such that $B(p,r)=\{p\}$ for all $p\in\fix(\phi)$. 
Consider $\xi$ with $T=T_0$ from Lemma \ref{lemFixOpen} and take 
$\delta=\min\{\alpha/2,\xi,r\}$. 
Suppose that $x,y\in \Lambda$, $s\colon\R\to\R$ is continuous, $s(0)=0$, 
\begin{equation}
\label{ecuSepHka}
\dist(\phi_t(x),\phi_{s(t)}(x))<\delta\text{ for all }t\in \R 
\end{equation}
and 
\begin{equation}
 \label{ecuSepHkb}
  \dist(\phi_t(x),\phi_{s(t)}(y))<\delta\text{ for all }t\in \R. 
\end{equation}
By \eqref{ecuSepHka}, \eqref{ecuSepHkb} and the triangular inequality we deduce 
\begin{equation}
 \label{ecuSepHkc}
 \dist(\phi_{s(t)}(x),\phi_{s(t)}(y))<2\delta\leq\alpha\text{ for all }t\in\R. 
\end{equation}
If $x\in\fix(\phi)$, from \eqref{ecuSepHkb} with $t=0$, we see that $y=x$ because $\delta\leq r$.
Assume that $x\notin\fix(\phi)$. 
We will show that $|s(t)-t|<T$ for all $t\in\R$. 
If this is not the case, as $s$ is continuous and $s(0)=0$, 
there is $u\in\R$ such that $|s(u)-u|=T$. 
Then, $\dist(\phi_{s(u)}(x),\phi_u(x))>\xi\geq\delta$, which contradicts \eqref{ecuSepHka}.
The condition $|s(t)-t|<T$ for all $t\in\R$ implies that $s$ is onto. 
The facts that $s$ is onto and that $\phi$ is separating, with separating constant $\alpha$, allow us to conclude from \eqref{ecuSepHkc} that $y\in\phi_\R(x)$. 
This proves that $\phi$ is KH-expansive.
\end{proof}

\subsection{BW-expansivity}
\label{secBWExp}Let $\rep$ be the set of all the increasing homeomorphisms $h\colon\R\to\R$ and 
$\repo=\{h\in\rep:h(0)=0\}$. 
A flow $\phi$ defined on a metric space $(\Lambda,\dist)$ is \emph{BW-expansive} \cite{BW,Flinn} if for all $\epsilon>0$ there is $\delta>0$ such that 
if $x,y\in\Lambda$, $h\in\repo$ and $\dist(\phi_t(x),\phi_{h(t)}(y))<\delta$ for all $t\in\R$ then $\phi_{h(t)}(y)\in\phi_{[-\epsilon,\epsilon]}(x)$ for all $t\in\R$.
Separating flows do not share some basic properties with BW-expansive flows. 
For instance, a BW-expansive flow on a compact space has at most a finite number of periodic orbits of period less than a given constant. 
But, a separating flow may present an uncountable set of periodic orbits with bounded period: 

\begin{exa}
\label{exaAnPer}
Consider the vector field $X(x,y)=(x^2+y^2)(-y,x)$ on $\R^2$ and the invariant annulus $\Lambda=\{(x,y)\in\R^2: 1\leq x^2+y^2\leq 4\}$.
The orbits are circles.
The key property of this flow is that the periods of different periodic orbits are different, which easily allows us to see that 
the flow is separating. Also, this example has vanishing topological entropy. 
Then, we see no link between the counting of periodic orbits and the topological entropy for a separating flow. 
For BW-expansive flows \cite{BW} this link exists and parallels the theory of expansive homeomorphisms \cite[Theorem 5]{BW}. 
Therefore, \cite[Proposition 3.2.14]{HK} cannot be applied to separating (KH-expansive) flows. 
Thus, the arguments given in \cite{Ham} for the proof of \cite[Corollary 3.3]{Ham} seems to be not sufficient.
\end{exa}

\section{Komuro and rescaled expansivity}
\label{secKomYRes}
In \S \ref{secKomExp} we consider the Komuro's version of expansivity and we prove an equivalence. 
In \S \ref{secEfficient} we introduce the notion of efficient flow and we give sufficient conditions for a flow to be efficient.
In \S \ref{secRescaling} we consider the rescaled expansivity of \cite{WW}.
For $a,b\in\R$ denote by $[a,b]$ the closed interval of real numbers between $a$ and $b$ (for $a\leq b$ and $a>b$).

\subsection{Komuro expansivity} In this section we give a characterization of Komuro's expansivity that will be used later.
\label{secKomExp}

\begin{df}[\cite{K}]
 A flow $\phi$ of a compact metric space $\Lambda$ is $k^*$-\emph{expansive} if for all $\epsilon>0$ there is 
 $\delta>0$ such that if 
\begin{equation}
\label{ecuAcomp}
x,y\in\Lambda, h\in\repo\text{ and }
\dist(\phi_t(x),\phi_{h(t)}(y))\leq\delta\text{ for all }t\in\R,  
\end{equation}
then $\phi_{h(t_0)}(y)\in\phi_{[-\epsilon,\epsilon]}(\phi_{t_0}(x))$ for some $t_0\in\R$.
\end{df}

\begin{rmk}
 Every $k^*$-expansive flow is separating. 
 Consequently, $\fix(\phi)$ is finite for a $k^*$-expansive flow of a compact metric space.
\end{rmk}

%
%
%
%

\begin{lem}
\label{lemRepKomExp}
Suppose that $x\notin\fix(\phi)$, $u\colon\R\to\R$ is a function, $h\in\rep$,
\begin{eqnarray}
\label{eqRescLem1}
\sup_{t\in\R}\diam(\phi_{[t,u(t)]}(x))<\diam(\phi_\R(x))/2\text{ and} \\
\label{eqRescLem2}\phi_{h(t)}(x)\in\phi_{[t,u(t)]}(x)\,\forall t\in\R.
\end{eqnarray}
Then there is $g\in\rep$ 
such that
$g(t)\in [t,u(t)]$ for all $t\in\R$ 
and $\phi_{g(t)}(x)=\phi_{h(t)}(x)$ for all $t\in\R$.
\end{lem}

\begin{proof}
If $x$ is not periodic then the result is trivial because the map $t\mapsto \phi_t(x)$ is injective and from 
\eqref{eqRescLem2} we can take $g=h$.

Suppose that $x\in\gamma$, a periodic orbit of period $\tau$.
By \eqref{eqRescLem2}
there is a function $k\colon \R\to\Z$ such that 
$g(t)=h(t)+\tau k(t)\in[t,u(t)]$ for all $t\in\R$.
We will show that $k$ is constant. 
Take $L>0$ such that
\[
 \sup_{t\in\R}\diam(\phi_{[t,u(t)]}(x))<L<\diam(\gamma)/2.
\]
Given $t\in\R$, define 
\begin{eqnarray*}
\tau_1=\inf\{s\geq t:\diam(\phi_{[t,s]}(x))=L\}\text{ and}\\
\tau_2=\sup\{s\leq t:\diam(\phi_{[t,s]}(x))=L\}.
\end{eqnarray*}
%
From \eqref{eqRescLem1} we have that $\tau_1-\tau_2<\tau$ and $\tau_2<u(t)<\tau_1$.
By the continuity of the flow, there is $\delta>0$ such that if $|r-t|<\delta$ then 
$\tau_2<u(r)<\tau_1$.
Therefore, $k(r)=k(t)$ whenever $|r-t|<\delta$. 
This implies that $k$ is constant.
\end{proof}

The next result will be used in Theorem \ref{thmKestImpRes} to connect $k^*$-expansivity with rescaled expansivity.

\begin{thm}
\label{thmCharK*}
For a flow $\phi$ on a compact metric space $(\Lambda,\dist)$ the following statements are equivalent:
\begin{enumerate}
 \item[a)] $\phi$ is $k^*$-expansive,
 \item[b)] for all $\beta>0$ there is 
 $\delta>0$ such that 
 \eqref{ecuAcomp} implies that $y=\phi_u(x)$ for some $u\in\R$ and $\diam(\phi_{[0,u]}(x))<\beta$,
 \item[c)] for all $\beta>0$ there is 
 $\delta>0$ such that 
 \eqref{ecuAcomp} implies 
 that there is $g\in\rep$ such that $\phi_{h(t)}(y)=\phi_{g(t)}(x)$ for all $t\in\R$ with $\diam(\phi_{[t,g(t)]}(x))<\beta$ for all $t\in\R$.
\end{enumerate}
\end{thm}

\begin{proof}
 The equivalence of a) and b) was proved in \cite[Theorem 1.3]{Ar}. 
 It is clear that c) implies b). Let us show that b) implies c). 
 Suppose that $\beta>0$ is given. 
 The condition b) implies that the fixed points are dynamically isolated and we can apply Lemma \ref{lemIsoBetaCero}
 to obtain $\beta_0\in (0,\beta)$ such that for every orbit $\gamma$ (not a fixed point) it 
 holds that $2\beta_0<\diam(\gamma)$.  
 Applying b), there is $\delta_0>0$ such that 
 if $x,y\in\Lambda$, $h\in\repo$ and 
 $\dist(\phi_t(x),\phi_{h(t)}(y))\leq\delta_0$ for all $t\in\R$ then   
 $y=\phi_u(x)$ for some $u\in\R$ and $\diam(\phi_{[0,u]}(x))<\beta_0$. 
 If this condition is applied to each pair $(\phi_t(x),\phi_{h(t)}(y))$, for each $t\in\R$, 
 we obtain a function $u\colon\R\to\R$ and we conclude that 
 $\diam(\phi_{[t,u(t)]}(x))<\beta_0$ and $\phi_{h(t)}(y)\in\phi_{[t,u(t)]}(x)$ for all $t\in\R$.
 By Lemma \ref{lemRepKomExp} there is $g\in\rep$ such that 
 $g(t)\in [t,u(t)]$ for all $t\in\R$ and
$\phi_{h(t)}(y)=\phi_{g(t)}(x)$ for all $t\in\R$.
%
\end{proof}

\subsection{Efficient flows} 
\label{secEfficient}
In this section we will assume that 
$\phi\colon\R\times M\to M$ is a flow with $C^1$ velocity field $X$, where
$M\subset \R^d$ is a compact $C^2$ manifold.
On $\R^d$ we consider the norm $\|\cdot\|$ induced by the usual inner product $\left<\cdot,\cdot\right>$.
The distance on $M$ is the restriction of $\dist(x,y)=\|y-x\|$ for all $x,y\in\R^d$. 
Assume that $\Lambda\subset M$ is invariant under $\phi$.
The restriction of $\phi$ to $\Lambda$ will be denoted as $\phi|_\Lambda$. 
\begin{df}
 \label{dfEfficient}
We say that $\phi|_\Lambda$ is \emph{efficient} 
if there is $\delta_*>0$ such that if 
 $0<\delta<\delta_*$, $x\in \Lambda\setminus\fix(\phi)$, $g\in\rep$ satisfy:
 \begin{itemize}
  \item $\dist(\phi_t(x),\phi_{g(t)}(x))<\delta\|X(\phi_t(x))\|$ for all $t\in\R$ and
  \item $\diam(\phi_{[t,g(t)]}(x))<\delta_*$ for all $t\in\R$
 \end{itemize}
then $\phi_{[t,g(t)]}(x)\subset B(\phi_t(x),\delta\|X(\phi_t(x))\|)$ for all $t\in\R$.
\end{df}

Technical details aside, the idea is that 
a flow is efficient if
for every small orbit arc $\phi_{[0,t]}(x)$, its diameter is controlled by the distance between the extreme points.
In this section we give sufficient conditions for a flow to be efficient.
Section \ref{secResNonEff} is about an example that is not efficient.

The following elementary result is essentially \cite[Exercise 14, p. 25]{DoCarmo}. 

\begin{lem}
\label{lemDoCarmo}
Let $\gamma\colon[a,b]\to\R^d$ be a $C^2$ curve such that $\|\dot\gamma(t)\|=1$ for all $t\in[a,b]$. 
\begin{itemize}
  \item If $a<s<b$ and $\dist(\gamma(a),\gamma(s))=\sup_{a\leq t\leq b}\dist(\gamma(a),\gamma(t))$ then 
  $\|\ddot\gamma(s)\|\geq1/\dist(\gamma(a),\gamma(s))$.
  \item If $R>0$, $\|\ddot\gamma(t)\|\leq 1/R$ for all $t\in[a,b]$ and $\diam(\gamma([a,b]))<R$ 
  then $\diam(\gamma([a,b]))=\dist(\gamma(a),\gamma(b))$.
 \end{itemize}
\end{lem}

\begin{proof}
%
%
Let $\varphi\colon [a,b]\to\R$ be defined as $\varphi(t)=\|\gamma(a)-\gamma(t)\|^2$. 
We have that $\varphi$ has a maximum at $t=s$. 
Then $\dot\varphi(s)=\left<\gamma(a)-\gamma(s),\dot\gamma(s)\right>=0$ and 
\[
 \ddot\varphi(s)=\left<\gamma(a)-\gamma(s),\ddot\gamma(s)\right>+\left<\dot\gamma(s),\dot\gamma(s)\right>\leq 0.
\]
Since $\|\dot\gamma(s)\|=1$ we have that 
$\left<\gamma(a)-\gamma(s),\ddot\gamma(s)\right>\leq-1$. 
Thus $$\|\gamma(a)-\gamma(s)\| \|\ddot\gamma(s)\|\geq 1.$$ 
This proves the first part. 

Now, arguing by contradiction, assume that $\diam(\gamma([a,b]))=\dist(\gamma(c),\gamma(s))$ with $a<s<b$. 
Then 
\[
 \frac 1{\dist(\gamma(c),\gamma(s))}\leq \|\ddot\gamma(s)\|\leq\frac 1 R
\]
This contradicts that $\diam(\gamma([a,b]))<R$ and the proof ends.
\end{proof}

Let $\kappa\colon M\setminus\fix(\phi)\to\R$ be the curvature of the regular
trajectory of $\phi$ at $x\in M$, $x\notin\fix(\phi)$.
It can be calculated as
\begin{equation}
\label{ecuCurvatura}
  \kappa(x)=\frac{\sqrt{\|d_xX(X(x))\|^2 \|X(x)\|^2-\left<d_xX(X(x)),X(x)\right>^2}}
 {\|X(x)\|^3}.
\end{equation}
See \cite[Exercise 12, p. 25]{DoCarmo}.
It is the curvature of the trajectory as a curve in $\R^d$. 

\begin{prop}
\label{propCurvEff} 
Let $\phi$ be a flow on a compact manifold $M$ with $C^1$ velocity field $X$ and 
an invariant subset $\Lambda$.
If the curvature $\kappa|_\Lambda$ is bounded then $\phi|_\Lambda$ is efficient.
\end{prop}

\begin{proof}
Take $R>0$ such that $\kappa(x)\leq 1/R$ for all $x\in \Lambda\setminus\fix(\phi)$. 
By Lemma \ref{lemDoCarmo} we see that $\phi$ is efficient with $\delta_*=1/R$.
\end{proof}

\begin{rmk}
 From Proposition \ref{propCurvEff} we have that flows with $C^1$ velocity fields 
 and without fixed points
 are efficient. Around a fixed point the curvature of the trajectories may not be bounded and 
 we need more hypothesis and arguments to conclude the efficiency of the flow. 
 This is what will be done in the remaining of this section.
\end{rmk}

Define 
\[
 m(T)=\inf\{\|T(v)\|:v\in \R^d, \|v\|=1\}
\]
for a linear transformation $T$.
Given a subset $P\subset M$ and $x\in M$ we define $\dist(x,P)=\inf_{p\in P}\dist(x,p)$.

\begin{lem}
\label{lemRescTiempo}
Let $\phi$ be a flow on a compact manifold $M$ with $C^1$ velocity field $X$.
 If $d_\sigma X$ is invertible for all $\sigma\in\fix(\phi)$ 
then 
there are constants $B,C>0$ such that
\begin{equation}
 \label{ecuRegBC}
 \frac 1C\dist(z,\fix(\phi))\leq\|X(z)\|\leq B\dist(z,\fix(\phi))
\end{equation}
for all $z\in M$. 
\end{lem}

\begin{proof}
Since $X$ is $C^1$, in particular it is Lipschitz, and there is $B>0$ satisfying the second inequality of \eqref{ecuRegBC}.
To prove the other part, take $\sigma\in\fix(\phi)$ and a local chart around $\sigma$. 
We can assume that $M=\R^n$ ($n=\dim(M)$) and $\sigma$ is the origin. 
Then $X(z)=d_\sigma X(z)+r(z)$. 
By hypothesis, $d_\sigma X$ is invertible and consequently $m(d_\sigma X)>0$.
Since $X$ is differentiable at $\sigma$, there is $r_\sigma>0$ such that if $\|z\|<r_\sigma$ then $\|r(z)\|/\|z\|<\frac 12 m(d_\sigma X)=C_\sigma$.
Thus 
\[
 \frac{\|X(z)\|}{\|z\|}\geq \|d_\sigma X(z/\|z\|)\|-\|r(z)\|/\|z\|\geq C_\sigma
\]
if $\|z\|<r_\sigma$. 
Note that $\|z\|$ in local charts is equivalent to $\dist(z,\sigma)$, the Riemannian distance on $M$.
Then, there are $C_1>0$ and an open neighborhood $U\subset M$ of $\fix(\phi)$ such that $\|X(z)\|/\dist(z,\fix(\phi))\geq C_1$ for all $z\in U\setminus \fix(\phi)$. 
Since $M\setminus U$ is compact, $\|X(z)\|/\dist(z,\fix(\phi))$ is bounded away from zero for $z\in M\setminus U$ and 
there is $C>0$ satisfying the first part of \eqref{ecuRegBC}
for all $z\in M$. 
\end{proof}

\begin{prop}
\label{propRescTiempo}
Let $\phi$ be a flow on a compact manifold $M$ with $C^1$ velocity field $X$. 
 If fixed points are dynamically isolated and 
 $d_\sigma X$ is invertible for all $\sigma\in\fix(\phi)$ 
then $\phi$ is efficient.
\end{prop}

\begin{proof}
Take $B,C$ from Lemma \ref{lemRescTiempo}.
We start showing that there is $A>0$ such that 
 \begin{equation}
 \label{ecuRegA}
\kappa(z)\dist(z,\fix(\phi))\leq A
 \end{equation}
for all $z\in M\setminus\fix(\phi)$. 
From \eqref{ecuCurvatura} we see that 
 $$\kappa(z)\leq\frac{\|d_zX(X(z))\| \|X(z)\|}{\|X(z)\|^3}=\frac{\|d_zX(X(z))\| }{\|X(z)\|^2}\leq \frac{\|d_zX\| }{\|X(z)\|}.$$
As $X$ is $C^1$ and $M$ is compact there is $A_1>0$ such that $\|d_zX\|\leq A_1$ for all $z\in M$. 
Then, define $A=A_1C$ and \eqref{ecuRegA} follows from \eqref{ecuRegBC}.
As singularities are dynamically isolated there is $r_2>0$ such that if $x\in M\setminus\fix(\phi)$ then $\dist(\phi_t(x),\fix(\phi))> r_2$ for some $t\in\R$.

To prove that $\phi$ is efficient 
define
\[
 \delta_*=\min\left\{\frac{r_2}2,\frac1{B(A+1)},
 \min_{\dist(z,\fix(\phi))\geq \frac{r_2}2} \frac 1{\kappa(z)}.
 \right\}
\]

Suppose that 
$0<\delta<\delta_*$, $x\in M\setminus\fix(\phi)$, $g\in\rep$, 
\begin{equation}
\label{ecuAcompLem}
 \dist(\phi_t(x),\phi_{g(t)}(x))<\delta\|X(\phi_t(x))\|\text{ for all }t\in\R
\end{equation}
and
$\diam(\phi_{[t,g(t)]}(x))<\delta_*$ for all $t\in\R$.

Since fixed points are isolated and $x$ is not a fixed point there is $s\in\R$ such that $\dist(\phi_s(x),\fix(\phi))\geq r_2$.
As $\diam(\phi_{[s,g(s)]}(x))<\delta_*\leq r_2/2$ 
we have that $\dist(z,\fix(\phi))>r_2/2$ for all $z\in \phi_{[s,g(s)]}(x)$ 
and consequently $\kappa(z)<1/\delta_*$ for all $z\in \phi_{[s,g(s)]}(x)$. 
By Lemma \ref{lemDoCarmo} we have that $\diam(\phi_{[s,g(s)]}(x))=\dist(\phi_s(x),\phi_{g(s)}(x))$. 
From \eqref{ecuAcompLem} we conclude that $\phi_{[s,g(s)]}(x)\subset B(\phi_s(x),\delta\|X(\phi_s(x))\|)$. 
It remains to prove this inclusion for all $s\in\R$.

Arguing by contradiction, suppose that there are $u\in\R$ and $z\in \phi_{[u,g(u)]}(x)$ such that 
$\dist(\phi_u(x),z)=\delta\|X(\phi_u(x))\|$ and 
$\phi_{[t,g(t)]}(x)\subset B(\phi_t(x),\delta\|X(\phi_t(x))\|)$ (open ball) whenever $|t-s|<|u-s|$.
Applying \eqref{ecuRegBC} we obtain
\[
 \begin{array}{lll}
  \dist(z,\fix(\phi))&\geq &\dist(\phi_u(x),\fix(\phi))-\dist(\phi_u(x),z)\\
  &\geq & \frac{1}{B} \|X(\phi_u(x))\| -\delta\|X(\phi_u(x)\|\\
  &=&\|X(\phi_u(x))\|(\frac 1B-\delta).
 \end{array}
\]
By \eqref{ecuRegA} we have 
\[
 \frac{A}{\kappa(z)}\geq \dist(z,\fix(\phi))\geq \|X(\phi_u(x))\|(\frac 1B-\delta).
\]
Lemma \ref{lemDoCarmo} implies that $\kappa(z)\geq\frac 1{\delta\|X(\phi_u(x))\|}$.
Then $A\delta\geq \frac 1B-\delta$, contradicting that $\delta<\delta_*\leq\frac{1}{B(A+1)}$.
This contradiction proves that 
$\phi_{[t,g(t)]}(x)\subset B(\phi_t(x),\delta\|X(\phi_t(x))\|)$ for all $t\in\R$ and $\phi$ is efficient. 
\end{proof}

In \S \ref{secResNonEff} we will see an example of a non-efficient flow 
with a fixed point that is not dynamically isolated.

\begin{cor}
\label{corRescTiempo}
Let $\phi$ be a flow on a compact manifold $M$ with $C^1$ velocity field $X$. 
 If fixed points are hyperbolic
then $\phi$ is efficient.
\end{cor}

\begin{proof}
It follows by Proposition \ref{propRescTiempo} because hyperbolic fixed points are
dynamically isolated and 
$d_\sigma X$ is invertible for every fixed point $\sigma$.
\end{proof}

\begin{rmk}
 A flow may have a fixed point $\sigma$ being dynamically isolated, non-hyperbolic and with $d_\sigma X$ invertible. 
 As an example consider $$X(x,y)=(y+x^3,-x+y^3)$$ 
 with $\sigma=(0,0)$. Thus, Proposition \ref{propRescTiempo} is strictly stronger than 
 Corollary \ref{corRescTiempo}.
\end{rmk}

\subsection{Rescaling expansivity}
\label{secRescaling}
For the next definition we will assume that $M$ is a Riemannian manifold 
with a flow $\phi$ having a $C^1$ velocity field $X$
and $\Lambda\subset M$ will denote a compact invariant subset.

\begin{df}[\cite{WW}] 
The flow $\phi|_\Lambda$ is \emph{rescaling expansive} if for all $\epsilon>0$ 
there is $\delta>0$ such that if 
\begin{equation}
\label{ecuAcompResc}
x,y\in\Lambda, h\in\repo\text{ and }\dist(\phi_t(x),\phi_{h(t)}(y))\leq\delta\|X(\phi_t(x))\|\text{ for all }t\in\R,  
\end{equation}
then $\phi_{h(t)}(y)\in\phi_{[-\epsilon,\epsilon]}(\phi_t(x))$ for all $t\in\R$.\footnote{ In \cite[Definition 1.1]{WW}, the definition of rescaled expansivity, 
$h$ is assumed to be 
 increasing and continuous but not surjective. However, as proved in  
 \cite[Corollary 2.6]{WW}, both forms are equivalent.}
\end{df}




The definition of rescaled expansivity was designed to accompany multisingular hyperbolicity \cite{WW}. 
The next example shows that without extra hypothesis the definition by itself is \emph{too much generous} with fixed points.

\begin{exa}
\label{exaNull}
Let $X=0$ be the vector field on a compact manifold $M$ without regular points. 
To prove that it is rescaling expansive, given any $\epsilon>0$ take $\delta=1$. 
If $\dist(\phi_t(x),\phi_{h(t)}(y))\leq\delta\|X(\phi_t(x))\|=0$ then $x=y$ (a fixed point). 
\end{exa}

\begin{rmk}
The Example \ref{exaNull} shows that a rescaling expansive flow may not be $k^*$-expansive, in fact may not even be separating (Remark \ref{rmkSepFix}).
Also, notice that the flow of Example \ref{exaAnPer} is separating but it is not rescaling expansive (because rescaled expansivity without fixed points is equivalent to BW-expansivity and 
compact surfaces admits no BW-expansive flows).
\end{rmk}


\begin{rmk}
 The definition of rescaled expansivity is independent of the Riemannian metric. 
 Therefore, without loss of generality we will assume that $M\subset \R^d$ and on $TM$ we consider the inner product of $\R^d$.
\end{rmk}

For our next theorem, a characterization of rescaled expansivity, we need the following result.

\begin{lem}[\cite{WW}]
\label{lem23WW}
If $\phi$ is a flow with $C^1$ velocity field $X$ on a compact Riemannian manifold $M$ then there is $r_0>0$ such that 
if $x\in M\setminus\fix(\phi)$, $0<\delta<r_0$ and $\phi_{[0,t]}(x)\subset B(x,\delta\|X(x)\|)$ then $|t|<3\delta$.
\end{lem}

See \cite[Lemma 2.3]{WW} for a proof of Lemma \ref{lem23WW}.

\begin{rmk}
By \cite[Theorem 3]{BodL} we have that every multisingular hyperbolic flow is a star flow, and in particular, 
the fixed points are hyperbolic and dynamically isolated.
From this viewpoint, it is natural to assume that fixed points are dynamically isolated. 
\end{rmk}

\begin{thm}
\label{thmResc}
Suppose that $\Lambda\subset M$ is an invariant compact subset and that the fixed points are dynamically isolated. 
If $\phi|_\Lambda$ is rescaling expansive then:
\begin{itemize}
\item[(*)] for all $\beta>0$ 
there is $\delta>0$ such that 
if $x,y\in\Lambda$, $h\in\repo$ and $\dist(\phi_t(x),\phi_{h(t)}(y))\leq\delta\|X(\phi_t(x))\|$ for all $t\in\R$
then
there is $g\in\rep$ such that $\phi_{h(t)}(y)=\phi_{g(t)}(x)$ for all $t\in\R$ and 
$\diam(\phi_{[t,g(t)]}(x))<\beta$ for all $t\in\R$.
\end{itemize}
The converse is true if the flow is efficient.
\end{thm}

\begin{proof}
\emph{Direct}. 
Consider $\beta_0$ given by Lemma \ref{lemIsoBetaCero}.
Suppose that $\beta\in (0,\beta_0/2)$ is given. Since $\Lambda$ is compact there is $\epsilon>0$ such that $\diam(\phi_{[-\epsilon,\epsilon]}(x))<\beta$ for all $x\in\Lambda$. 
The definition of rescaled expansivity implies that there is $\delta>0$ 
such that if $x,y\in\Lambda$, $h\in\repo$ and
$\dist(\phi_t(x),\phi_{h(t)}(y))\leq\delta\|X(\phi_t(x))\|$ for all $t\in\R$ then 
$\phi_{h(t)}(y)\in\phi_{[-\epsilon,\epsilon]}(\phi_t(x))$ for all $t\in\R$.
Thus 
$$\diam(\phi_{[-\epsilon,\epsilon]}(\phi_t(x)))<\beta<\beta_0/2<\diam(\phi_\R(x))/2$$
for all $t\in\R$. 
By Lemma \ref{lemRepKomExp},
there is $g\in\rep$ 
such that
$g(t)\in [t,u(t)]$ for all $t\in\R$ 
and $\phi_{g(t)}(x)=\phi_{h(t)}(x)$ for all $t\in\R$.

%

\emph{Converse}. 
To prove that $\phi$ is rescaling expansive consider $\epsilon>0$ given. 
Take $r_0$ from Lemma \ref{lem23WW}, 
$\delta_*$ from Definition \ref{dfEfficient} (efficient flow)
and define 
$$\beta=\min\{\epsilon,\delta_*,r_0\}.$$ 
From the hypothesis (*), take the corresponding value of $\delta$ associated to $\beta$. 
Let $\delta_0=\min\{\delta,\epsilon/3,r_0,\delta_*\}$ and suppose that  
$x,y\in\Lambda$, $h\in\repo$ and 
$$\dist(\phi_t(x),\phi_{h(t)}(y))\leq\delta_0\|X(\phi_t(x))\|$$ 
for all $t\in\R$. 
By (*) we know that there is $g\in\rep$ such that $\phi_{h(t)}(y)=\phi_{g(t)}(x)$ for all $t\in\R$ and 
$\diam(\phi_{[t,g(t)]}(x))<\beta$ for all $t\in\R$.
As the flow is efficient we have that 
$\phi_{[t,g(t)]}(x)\subset B(\phi_t(x),\delta_0\|X(\phi_t(x))\|)$ for all $t\in\R$.
%
By Lemma \ref{lem23WW} we have that $|g(t)-t|<3\delta_0\leq \epsilon$.
This implies that $\phi_{g(t)}(x)=\phi_{h(t)}(y)\in\phi_{[-\epsilon,\epsilon]}(\phi_t(x))$ for all $t\in\R$.
Then, $\phi$ is rescaling expansive.
\end{proof}

\subsection{A rescaled metric}
\label{secRescMetric}

Consider a Riemannian manifold $(M,\left<\cdot,\cdot\right>)$.
Let $\phi$ be a flow on $M$ with velocity field $X$. 
Given a subset $\Lambda\subset M$ define $\Lambda^*=\Lambda\setminus\fix(\phi)$.
Define the \emph{rescaled Riemannian metric} $\left<\cdot,\cdot\right>_r$ on $M^*$ as
\[
 \left<v,w\right>_r=\frac{\left<v,w\right>}{\left<X(p),X(p)\right>}
\]
where $v,w\in T_pM$.
Assuming that $M^*$ is connected, denote by $\dist_r$ the distance on $M^*$ induced by $\left<\cdot,\cdot\right>_r$.

\begin{thm}
\label{thmBWexprResc}
If $d_\sigma X$ is invertible for all $\sigma\in\fix(\phi)$ 
and $\phi|_{\Lambda^*}$ is BW-expansive with respect to $\dist_r$ then $\phi|_\Lambda$ is rescaling expansive.
\end{thm}

\begin{proof}
Given $\epsilon>0$ consider $\delta>0$ from the definition of BW-expansivity. 
By Lemma \ref{lemRescTiempo}
there are constants $B,C>0$ such that
\[
 \frac 1C\dist(z,\fix(\phi))\leq\|X(z)\|\leq B\dist(z,\fix(\phi))
\]
for all $z\in M$. 
Take $\rho>0$ such that $BC\rho/(1-\rho B)<\delta$.

We will show that 
\begin{equation}
 \label{eqDistrrhodelta}
\text{ if }\dist(p,q)\leq \rho\|X(p)\|\text{ then }\dist_r(p,q)\leq\delta. 
\end{equation}
Suppose that $\dist(p,q)\leq\rho\|X(p)\|$. Then
\[
\frac 1B\|X(p)\|-\rho\|X(p)\|\leq\dist(p,\fix(\phi))-\dist(p,q)\leq\dist(q,\fix(\phi))\leq C\|X(q)\|
\]
and, if $k=C(\frac 1B-\rho)^{-1}$ then
\[
\|X(p)\|\leq k\|X(q)\|
\]
Fix $a$ such that $\dist(p,a)\leq\rho\|X(p)\|$
and take a geodesic arc $\gamma$ (of the original metric of $M$) from $p$ to $a$, with $\|\dot\gamma\|=1$. 
Then 
\[
\begin{array}{ll}
 \dist_r(p,a)&\displaystyle\leq\int_0^{\dist(p,a)}\frac{\|\dot\gamma(t)\|}{\|X(\gamma(t))\|}dt=\int_0^{\dist(p,a)}\frac{1}{\|X(\gamma(t))\|}dt\\
             &\displaystyle\leq\int_0^{\dist(p,a)}\frac{k}{\|X(p)\|}dt=\frac{k\dist(p,a)}{\|X(p)\|}\\
             &\displaystyle\leq\frac{k\rho\|X(p)\|}{\|X(p)\|}=k\rho=\frac{BC\rho}{1-\rho B}<\delta.
\end{array}
\]
This proves \eqref{eqDistrrhodelta}.

To prove that $\phi|_\Lambda$ is rescaling expansive take $x,y\in\Lambda$ and $h\in\repo$ such that 
$\dist(\phi_t(x),\phi_{h(t)}(y))\leq \rho\|X(\phi_t(x))\|$ for all $t\in\R$.
Then $\dist(\phi_t(x),\phi_{h(t)}(y))\leq \delta$ for all $t\in\R$ and by the BW-expansivity of $\phi$ 
we have that $\phi_{h(t)}(y)\in\phi_{[-\epsilon,\epsilon]}(\phi_t(x))$ for all $t\in\R$, and the proof ends.
\end{proof}

Conversely, it would be interesting to know under what conditions 
the fact that $\phi|_\Lambda$ is rescaling expansive implies
that $\phi|_{\Lambda^*}$ is BW-expansive with respect to $\dist_r$.

\subsection{A rescaling expansive flow that is not efficient.}
\label{secResNonEff}
 On $\R^2$ consider the flow $\phi$ with velocity field $X(x,y)=(-y,x)$. 
 Define 
 \[
  \Lambda=\{p\in\R^2:\|p\|=e^{-n}, n\in\Z^+\}\cup\{(0,0)\}.
 \]
The next results are about this particular flow $\phi$ restricted to this compact invariant set $\Lambda$.
\begin{prop}
The flow $\phi|_\Lambda$ is rescaling expansive.
\end{prop} 

\begin{proof}
Consider the cylinder $N=\R\times(\R/\Z)$ with the product Riemannian metric. 
Define $\varphi\colon N\to\R^2\setminus\{(0,0)\}$ as 
$\varphi(r,\theta)=e^r(\cos\theta,\sin\theta)$. 
It is a diffeomorphism, moreover, if we consider the rescaled metric \S \ref{secRescMetric} associated to $X$ 
on $\R^2\setminus\{(0,0)\}$ we have that 
$\varphi$ is an isometry. 
With respect to the rescaled metric, $\Lambda\setminus\{(0,0)\}$ is a countable union of equidistant circles. 
Since each circle is a periodic orbit, we have that $\phi|_\Lambda$ is BW-expansive.
As $d_{(0,0)}X$ is invertible, we can apply Theorem \ref{thmBWexprResc} to conclude that $\phi|_\Lambda$ is rescaling expansive.
\end{proof}

\begin{prop}
 The flow $\phi|_\Lambda$ is not efficient.  
\end{prop}

%
%
%
We say that $\phi|_\Lambda$ is \emph{efficient} 
if there is $\delta_*>0$ such that if 
 $0<\delta<\delta_*$, $x\in \Lambda\setminus\fix(\phi)$, $g\in\rep$ satisfy:
 \begin{itemize}
  \item $\dist(\phi_t(x),\phi_{g(t)}(x))<\delta\|X(\phi_t(x))\|$ for all $t\in\R$ and
  \item $\diam(\phi_{[t,g(t)]}(x))<\delta_*$ for all $t\in\R$
 \end{itemize}
then $\phi_{[t,g(t)]}(x)\subset B(\phi_t(x),\delta\|X(\phi_t(x))\|)$ for all $t\in\R$.

\begin{proof}
Given $\delta_*>0$, suppose that $0<\delta<\delta_*/2$ and define $g(t)=2\pi+t$. 
Since the regular orbits of $\phi$ have period $2\pi$, we have that
$\dist(\phi_t(\delta,0),\phi_{g(t)}(\delta,0))=0<\delta\|X(\phi_t(\delta,0))\|$ for all $t\in\R$. 
 Also, $\diam(\phi_{[t,g(t)]}(\delta,0))=2\delta<\epsilon_0$ for all $t\in\R$.
But $\phi_{[t,g(t)]}(\delta,0)=\{p\in\R^2:\|p\|=\delta\}$ which is not contained in $B(\phi_t(\delta,0),\delta\|X(\phi_t(\delta,0))\|)$.
Thus, $\phi$ is not efficient.
\end{proof}

The previous result shows that in Proposition \ref{propRescTiempo} it is necessary to assume that fixed points are dynamically isolated.

\subsection{Komuro and rescaled expansivity}
Now we can prove the main result of the paper.
\begin{thm}
\label{thmKestImpRes}
If $\phi|_\Lambda$ is $k^*$-expansive and efficient then it 
is rescaling expansive. 
\end{thm}

\begin{proof}
We will prove that $\phi|_\Lambda$ is rescaling expansive applying the equivalence given in 
Theorem \ref{thmResc}.  
Suppose $\beta>0$ given. 
Since $\phi|_\Lambda$ is $k^*$-expansive, by Theorem \ref{thmCharK*} (item c)
there is $\delta_1>0$ such that 
if $x,y\in\Lambda$, $h\in\repo$ and 
$\dist(\phi_t(x),\phi_{h(t)}(y))\leq\delta_1$ for all $t\in\R$,  
then there is $g\in\rep$ such that $\phi_{h(t)}(y)=\phi_{g(t)}(x)$ for all $t\in\R$ with $\diam(\phi_{[t,g(t)]}(x))<\beta$ for all $t\in\R$.
Consider 
\begin{equation}
 \label{ecuCOndDelta}
  \delta=\frac{\delta_1}{\max_{z\in\Lambda}\|X(z)\|}
\end{equation}
Suppose that 
$x,y\in\Lambda$, $h\in\repo$ and 
$
\dist(\phi_t(x),\phi_{h(t)}(y))\leq\delta\|X(\phi_t(x))\|
$
for all $t\in\R$. 
From \eqref{ecuCOndDelta} we have that
$
\dist(\phi_t(x),\phi_{h(t)}(y))\leq\delta_1
$
for all $t\in\R$.
As we said, by Theorem \ref{thmCharK*} there is $g\in\rep$ such that 
$\phi_{h(t)}(y)=\phi_{g(t)}(x)$
for all $t\in\R$ with $\diam(\phi_{[t,g(t)]}(x))<\beta$ for all $t\in\R$. 
From Theorem \ref{thmResc} we conclude that $\phi|_\Lambda$ is rescaling expansive.
\end{proof}

\begin{cor}
 If the curvature of the regular trajectories is bounded and $\phi|_\Lambda$ is $k^*$-expansive then $\phi|_\Lambda$ is 
 rescaling expansive.
\end{cor}

\begin{proof}
 It follows by Proposition \ref{propCurvEff} (if the curvature of the trajectories is bounded then the flow is efficient)
 and Theorem \ref{thmKestImpRes}.
\end{proof}

\begin{cor}
\label{corFixHypKomRes}
Let $\phi$ be a flow with $C^1$ velocity field $X$ 
such that each $\sigma\in\fix(\phi)$ is hyperbolic. 
 If $\phi|_\Lambda$ is $k^*$-expansive 
 then $\phi|_\Lambda$ is rescaling expansive. 
\end{cor}

\begin{proof}
Note that if $\sigma$ is a hyperbolic singularity then $d_\sigma X$ is invertible as it has no vanishing eigenvalue.
Since every hyperbolic fixed points is dynamically isolated and 
$d_\sigma X$ is invertible for all $\sigma\in\fix(\phi)$, 
applying Proposition \ref{propRescTiempo} we conclude that 
$\phi$ is efficient on $M$.
By Theorem \ref{thmKestImpRes}, $\phi|_\Lambda$ is rescaling expansive. 
\end{proof}

We refer the reader to \cite{WW} for the missing definitions involved in the next result.

\begin{cor}
 There is a residual subset $R$ of the space of $C^1$ vector fields of the compact manifold $M$ such that 
 for all $X\in R$ it holds that: 
 if $\phi|_\Lambda$ is $k^*$-expansive then $\phi|_\Lambda$ is rescaled expansive, locally star and multisingular hyperbolic, where $\phi$ is the flow induced by $X$.
\end{cor}

\begin{proof}
 It follows by Corollary \ref{corFixHypKomRes} and \cite[Theorem B]{WW} because the set of vector fields with all of its fixed points hyperbolic is open and dense in the set of $C^1$ vector fields of $M$.
\end{proof}

\section{Flows of surfaces}
\label{secFlowOfSurf}

Let $\phi\colon\R\times S\to S$ be a flow on the compact surface $S$. Let us introduce some terminology of surface flows. 
Suppose that $\sigma\in\fix(\phi)$ is dynamically isolated.  
If $x\in S$ is a regular point such that $\phi_t(x)\to \sigma$ as $t\to +\infty$ (resp. $t\to -\infty$)
then the orbit of $x$ is a \emph{stable} (resp. \emph{unstable}) \emph{separatrix} of $\sigma$. 
We say that $\sigma$ is a fixed point of \emph{saddle type} if it has a positive and finite number of separatrices.
The \emph{index} of a fixed point $\sigma$ of saddle type is 1-$n_s$, where $n_s$ is the number of stable separatrices of $\sigma$.
Consider two flows $\phi$ and $\psi$ defined on the same surface $S$ and
let $\sigma\in\fix(\phi)$ be a fixed point of index 0.
Suppose that:
\begin{enumerate}
\item $\fix(\psi)=\fix(\phi)\setminus\{\sigma\}$,
\item if $x\notin\psi_\R (\sigma)$ then
$\phi_\R (x)=\psi_\R (x)$ and
\item the direction of both flows coincide on each orbit.
\end{enumerate}
In this case we say that $\psi$ \emph{removes} the fixed point $\sigma$ of $\phi$
and that $\phi$ \emph{adds} a fixed point to $\psi$.
We say that two flows defined on the same space are related by a \emph{time change} if they have the same orbits with the same orientation. 
A flow is \emph{strongly separating} \cite{ArK} if every time change is separating.

We will need the following results for a flow $\phi$ of a compact surface $S$: 
\begin{itemize}
\item \cite[Theorem 3.8]{ArK}: $\phi$ is strongly separating 
if and only if the fixed points are of saddle type and the union of their separatrices is dense in $S$. 
\item \cite[Lemma 4.1]{Ar}: Suppose that $\phi$ has not periodic points, $\fix(\phi)$ is a finite set,
the union of the stable separatrices is not dense
in the surface and $\Omega(\phi)=S$. Then $S$ is the torus and $\phi$ is an irrational flow.
\item \cite[Theorem 5.3]{Ar}: If $\psi$ removes a fixed point of $\phi$ then $\psi$ is $k^*$-expansive
if and only if $\phi$ is $k^*$-expansive.
\item \cite[Theorem 6.1]{Ar}: If $\phi$ has not fixed points of index 0 then the following statements are equivalent: 
\begin{itemize}
\item $\phi$ is $k^*$-expansive, 
\item the fixed points are of saddle type and the union of its separatrices is dense
in $S$, 
\item $\fix(\phi)$ is a finite and non empty set, $\Omega(\phi)=S$ and $\phi$ has no periodic points.
\end{itemize}
\end{itemize}
As usual $\Omega(\phi)$ denotes the non-wandering set of $\phi$.

\begin{thm}
\label{thmSupStSepRes}
Let $\phi$ be a flow of a compact surface with $C^1$ velocity field. 
If $\phi$ is strongly separating and efficient then it
is rescaling expansive.
\end{thm}

\begin{proof}
If the flow is $k^*$-expansive then the result follows from Theorem \ref{thmKestImpRes}.

Suppose that $\phi$ is strongly separating but not $k^*$-expansive. 
Let $\psi$ be a flow of $S$ removing the fixed points of $\phi$ with index 0.
By \cite[Theorem 5.3]{Ar} we know that $\psi$ is not $k^*$-expansive.
By construction, $\fix(\psi)$ is finite (possibly empty), $\Omega(\psi)=S$, and the fixed points (if any) of $\psi$ are of saddle type. 
Applying \cite[Theorem 6.1]{Ar} we deduce that the union of its separatrices cannot be dense
in $S$. If $\psi$ has a periodic orbit $\gamma$, then the orbits close to $\gamma$ are periodic because $\psi$ has not wandering points. 
This easy contradicts that $\phi$ is strongly separating (as we can take a time change of $\phi$ with a cylinder of 
infinitely many periodic orbits with the same period). Therefore, we conclude that $\psi$ has no periodic orbit and we can apply 
\cite[Lemma 4.1]{Ar} to see that $S$ is the torus and $\psi$ is an irrational flow.
Equivalently, $\phi$ is obtained from an irrational flow on the torus by adding fixed points of index 0.

We will show that $\phi$ is rescaling expansive applying Theorem \ref{thmResc}. 
Assume that $\psi$, the flow removing the fixed points of $\phi$, is a suspension of an irrational rotation of the circle.
Let $x,y\in M$ be a regular points and suppose that $\dist(\phi_t(x),\phi_{h(t)}(y))\leq\delta\|X(\phi_t(x))\|$ for all $t\in\R$, where $X$ is 
the velocity field of $\phi$. 
Take $t_n\in\R$ such that $\phi_{t_n}(x)\to \sigma$ for some $\sigma\in\fix(\phi)$. 
In particular, $\|X(\phi_{t_n}(x))\|\to 0$ and 
\begin{equation}
 \label{ecuSupStSep}
 \dist(\phi_{t_n}(x),\phi_{h(t_n)}(y))\to 0.
\end{equation}
Since $\psi$ is the suspension of an irrational rotation, 
if $x$ and $y$ were not in the same local orbit then $\inf_{t\in\R}\dist(\phi_t(x),\phi_{h(t)}(y))>0$ 
(provided that $\delta$ is sufficiently small). This would contradict \eqref{ecuSupStSep}. Then $a$ and $y$ are in the same local orbit, which proves that $\phi$ is rescaling expansive.
\end{proof}

\begin{rmk}
In Example \ref{exaToroUnaSing} we considered an irrational flow on the torus with a fixed point of index 0. 
Denote by $\phi$ this flow. 
We know that it is strongly separating by \cite[Theorem 3.8]{ArK}. 
Since the regular trajectories of $\phi$ are reparameterized geodesics, 
they have vanishing curvature and 
by Proposition \ref{propCurvEff} we conclude that the flow is efficient.
Applying Theorem \ref{thmSupStSepRes} we have that $\phi$ is rescaling expansive.
By \cite[Theorem 6.5]{Ar} we know that $\phi$ is not $k^*$-expansive because the two-dimensional torus does not admit 
$k^*$-expansive flows.
That is, a rescaling expansive flow with a finite number of singular points may not be $k^*$-expansive.
\end{rmk}

\subsection{An example on the sphere}
\label{secExSph}
Let $S\subset \R^3$ be the sphere given by $x^2+y^2+z^2=1$ and consider the 
vector field $X$ on $TS$ defined as 
\[
 X(x,y,z)=(x^2+y^2)(xz,yz,-x^2-y^2)
\]
for $(x,y,z)\in S$. 

\begin{prop}
The flow $\phi$ induced by $X$ is rescaling expansive on $S$. 
\end{prop}
\begin{proof}
The flow $\phi$ has two fixed points, a repeller $\sigma_1=(0,0,1)$ and an attractor $\sigma_2=(0,0,-1)$. 
The dynamics is simple, the trajectories goes from $\sigma_1$ to $\sigma_2$ and are contained in vertical planes.
If $p=(x,y,z)$ define $\rho(p)=\sqrt{x^2+y^2}$.
Since $\dist(p,\sigma_i)$ is equivalent to $\rho(p)$, around $\sigma_i$, we have that
\begin{equation}
 \label{ecuEjEsf}
 \lim_{p\to\sigma_i}\frac{\|X(p)\|}{\rho(p)}=\lim_{p\to\sigma_i}\frac{\|X(p)\|}{\dist(p,\sigma_i)}=0.
\end{equation}
We see, Lemma \ref{lemRescTiempo}, that $d_{\sigma_i} X$ is not invertible, and the fixed points are not hyperbolic.
Consider $r_0$ from Lemma \ref{lem23WW}.
Given a regular point $p\in S$ let $\pi(p)$ be the half-plane containing the orbit of $p$ and with boundary the $z$-axis.
If $\pi(p)\neq\pi(q)$ then there is $k>0$, depending on the angle between the planes, such that 
$\dist(p',\pi(q))>k\rho(p')$ for all $p'\in\pi(p)$.

Suppose that for some $\delta\in (0,r_0)$, $p,q\in S$ and $h\in\repo$ it holds that 
$\dist(\phi_t(p),\phi_{h(t)}(q))\leq\delta\|X(\phi_t(p))\|$.
If $\pi(p)\neq\pi(q)$ then
\[
 k\rho(\phi_t(p))<\dist(\phi_t(p),\phi_{h(t)}(q))\leq\delta\|X(\phi_t(p))\|
\]
for all $t\in\R$.
Taking the limit $t\to+\infty$ we arrive to a contradiction with \eqref{ecuEjEsf}.
This proves that $\pi(p)=\pi(q)$, i.e., $p$ and $q$ are in the same orbit. 
Since the trajectories of the flow are (reparameterized) geodesics, we have that 
$\phi_{[t,h(t)]}(p)\subset B(x,\delta\|X(x)\|)$. 
By Lemma \ref{lem23WW}, we conclude that $|t-h(t)|<3\delta$, and $\phi$ is rescaling expansive. 
\end{proof}

\begin{rmk}
\label{rmkResNonInv}
With similar techniques it can be proved that the flow of $Y(x,y,z)=(xz,yz,-x^2-y^2)$ is not rescaling expansive. 
Essentially, this is because 
$$\lim_{p\to\sigma_i}\frac{\|Y(p)\|}{\rho(p)}=1.$$
Then, the property of rescaled expansivity is not invariant under time changes of the flow.
Note that the flows of $X$ and $Y$ are not separating.
\end{rmk}
%
%
%
%
%

\end{document}